\documentclass{amsproc}

\usepackage{}

\newtheorem{theorem}{theorem}[section]
\newtheorem{lem}[theorem]{Lemma}
\newtheorem{thm}[theorem]{Theorem}

\newtheorem{prop}[theorem]{Proposition}

\theoremstyle{definition}

\theoremstyle{remark}

\numberwithin{equation}{section}

\begin{document}

\title[Cyclic Algebras of Odd Prime Degree]{Cyclic Division Algebras of Odd Prime Degree are never Amitsur-Small}


\author{Adam Chapman}
\address{School of Computer Science, Academic College of Tel-Aviv-Yaffo, Rabenu Yeruham St., P.O.B 8401 Yaffo, 6818211, Israel}
\email{adam1chapman@yahoo.com}
\thanks{}

\author{Ilan Levin}
\address{Department of Mathematics, Bar-Ilan University, Ramat Gan, 55200}
\email{ilan7362@gmail.com}

\author{Marco Zaninelli}
\address{Department of Mathematics, University of Pennsylvania, 
Philadelphia, 
USA}
\email{zaninelli.marco21@gmail.com}

\date{}

\begin{abstract}
A division ring $D$ is Amitsur-Small if for every $n$ and every maximal left ideal $I$ in $D[x_1,\dots,x_n]$, $I \cap D[x_1,\dots,x_{n-1}]$ is maximal in $D[x_1,\dots,x_{n-1}]$. The goal of this note is to prove that cyclic division algebras of odd prime degree over their center are never Amitsur-Small.
\end{abstract}

\maketitle

\section{Introduction}

It is a classical result that for a field $F$ and a maximal ideal $I$ in $F[x_1,\dots,x_n]$, the ideal $I \cap F[x_1,\dots,x_{n-1}]$ is maximal in $F[x_1,\dots,x_{n-1}]$. In \cite{AmitsurSmall:1978}, Amitsur and Small asked the analogous question for division rings, i.e., if $D$ is a division ring and $I$ is a maximal left ideal in $D[x_1,\dots,x_n]$, is $I \cap D[x_1,\dots,x_{n-1}]$ necessarily a maximal left ideal in $D[x_1,\dots,x_{n-1}]$.
In \cite{ChapmanParan:2025}, the authors defined  a division ring $D$ to be Amitsur-Small if for any $n \geq 2$ and any maximal left ideal $I$ in $D[x_1,\dots,x_n]$, the answer to this question is positive.
In the same paper, the authors showed that $\mathbb{H}$ is Amitsur-Small, but any quaternion division algebra $D$ over a non-Pythagorean field is not Amitsur-Small. The authors also showed that degree $3$ division algebras are never Amitsur-Small. 

The goal of this note is to complete the picture for cyclic division algebras of odd prime degree, by showing that they are never Amitsur-Small.

It is important to note that in this context (like in \cite{Rowen:1992}), the polynomial ring $D[x_1,\dots,x_n]$ is defined to be $D \otimes_F F[x_1,\dots,x_n]$, and thus, the variables $x_1,\dots,x_n$ are central in this ring. In the language of \cite{GM}, this is denoted by $D_L[x_1,\dots,x_n]$, not to be confused with $D_G[x_1,\dots,x_n]$, which is a much larger ring.

\section{Cyclic Algebras}
A cyclic algebra $D=(K/F,\beta)$ of degree $n$ over a field $F$ is generated by two elements $i$ and $j$ such that $K=F[i]$, $K/F$ is cyclic of degree $n$ with Galois group $\langle \sigma \rangle$, where $j t j^{-1}=\sigma(t)$ for any $t\in F[i]$, and $j^n=\beta\in F^\times$.
Division rings of prime degree over their center are conjectured to always be cyclic, but in this note, we simply assume that the algebras under discussion are cyclic.
For a good reference on cyclic algebras, see \cite[Section 2.5]{GilleSzamuely:2017}.

\begin{lem}\label{Maximal subfields}
Let $D$ be a cyclic algebra of odd prime degree $p$ over $F$. If $F[i]$ is isomorphic to $F[j]$ as $F$-algebras, then $D$ is not a division algebra.
Consequently, if $D$ is a division algebra, then $F[i] \not \cong F[j]$.
\end{lem}

\begin{proof}
There is a famous result (see \cite[Page 98, Theorem 19]{Albert:1968}) that $D$ is a division algebra if and only if $\beta$ is not a norm in the field extension $F[i]/F$. If $F[j] \cong F[i]$ as $F$-algebras, then $F[i]$ contains the image of $j$ under this isomorphism, and the norm of this element is exactly $\beta$, and thus $D$ is not a division algebra.
\end{proof}

Let $p$ be a prime number. We recall that a $p$-special field is a field $L$ for which every finite field extension $M/L$ has $\deg(M/L)=p^m$ for some nonnegative integer $m$. For example, $\mathbb{R}$ is $2$-special.
By \cite{Lotscher:2013}, a necessary and sufficient condition for being $p$-special is not having nontrivial extensions of prime-to-$p$ degree.
Therefore, a $p$-special closure $L$ of $F$ is obtained by iteratively extending scalars to all prime-to-$p$ extensions.
Consequently, a division $F$-algebra of degree $p$ remains a division algebra under scalar extension to its $p$-special closure $L$.

\begin{lem}\label{F[j] factorization}
Let $D$ be a cyclic algebra of odd prime degree $p$ over $F$. Let $f$ be  the minimal polynomial of $i$. Then $f$ has no proper monic right-hand factor $g$ whose coefficients are in $F[j]$, and thus, no such factor that commutes with $j$.
\end{lem}

\begin{proof}
As shown above, there exists a $p$-special field $L$ containing $F$ obtained by taking the closure under all prime-to-$p$ extensions.
The resulting algebra $D\otimes_F L$ is still a division algebra, and its maximal subfields $L[i]$ and $L[j]$, are thus not isomorphic.
Now, assume the contrary, that $f$ has such a proper right-hand factor $g$ with coefficients in $F[j]$. 
Its coefficients lie in $L[j]$, which is also $p$-special, and thus $g$ has a root in $L[j]$.
This root is conjugate to $i$ by Wedderburn's Theorem (see \cite{Rowen:1992}), and so $L[j]$ contains an isomorphic copy of $L[i]$. 
Since both $L[i]/L$ and $L[j]/L$ are of degree $p$, we have that $L[j] \cong L[i]$, which contradicts Lemma \ref{Maximal subfields}.
\end{proof}

\section{Amitsur-Small}

In this section we show explicitly why cyclic division algebras of odd prime degree are not Amitsur-Small.
We first recall a useful result from \cite{ChapmanParan:2025}:

\begin{prop}[{\cite[Proposition 2.3]{ChapmanParan:2025}}]\label{useful}
Let $D$ be a division algebra, let $f\in D[x]$ be a monic reducible polynomial, and let $a\in D$ be an element such that $af=fa$ and every proper monic right-hand factor of $f$ does not commute with $a$. Then the left ideal $M$ generated by $f,y-a$ in $D[x,y]$ is a maximal left ideal, whose intersection with $D[x]$ is not a maximal left ideal in $D[x]$.
\end{prop}

\begin{thm}
If $D$ is a cyclic division algebra of odd prime degree $p$ over a field $F$, then $D$ is not Amitsur-Small.
\end{thm}

\begin{proof}
It is enough to consider $D[x,y]$ and find a left ideal $I$ which is maximal there and such that $I \cap D[x]$ is not maximal.
As in Section 2, let $i$ and $j$ be a standard pair of generators of $D$, and take $I = \langle f,y-j \rangle$ where $f$ is the minimal polynomial of $i$.
By Lemma \ref{F[j] factorization}, $j$ is an element that commutes with $f$ but does not commute with any monic proper right-hand factor of $f$.
Therefore, by Proposition \ref{useful}, $I$ satisfies all the requirements, and the statement follows.
\end{proof}

\section*{Acknowledgements}
The authors thank the referee for the careful reading and useful feedback.
\bibliographystyle{amsplain}
\def\cprime{$'$}

\end{document}